\documentclass[10 pt,a4paper,reqno]{amsart}

\usepackage[utf8]{inputenc}

\usepackage{array,amsbsy,amscd,amsfonts,color,amssymb,amstext,amsmath,latexsym,amsthm,amscd,multicol,graphicx, tikz, physics}
\usepackage[english]{babel}
\usepackage{hyperref}
\usepackage[mathscr]{eucal}
\usepackage{tikz-cd}

\makeatletter
\@namedef{subjclassname@1991}{Subject}
\makeatother

\newtheorem{theorem}{Theorem}
\numberwithin{theorem}{section}

\theoremstyle{definition}
\theoremstyle{definition}\newtheorem{definition}[theorem]{Definition}
\theoremstyle{definition}
\theoremstyle{definition}\newtheorem{example}[theorem]{Example}
\theoremstyle{definition}
\theoremstyle{definition}
\theoremstyle{definition}
\theoremstyle{definition}
\theoremstyle{definition}
\theoremstyle{definition}

\title{A caricature of dilation theory}
\author{B.V. Rajarama Bhat}
\address{Indian Statistical Institute, Stat-Math. Unit, R V College Post, Bengaluru 560059, India}
\email{bhat@isibang.ac.in}
\author{Sandipan De}
\email{444sandipan@gmail.com }
\author{ Narayan Rakshit}
\email{narayan753@gmail.com   }


\begin{document}

    \begin{abstract}
        We present a set-theoretic version of some basic dilation results of operator theory.  The results we have considered are Wold decomposition, Halmos dilation, Sz. Nagy
        dilation, inter-twining lifting, commuting and non-commuting dilations, BCL theorem etc. We point out some natural generalizations and variations.
    \end{abstract}

    \maketitle

The basic aim of dilation theory of  Hilbert space operators is to realize operators which are a priori not very tractable as compressions of better-behaved operators such as isometries
or unitaries.  This is a very well developed subject with a number of applications (See \cite{SzF}, \cite{Pa}) and is also a motivation for studying dilation of
quantum dynamical semigroups (See \cite{Ar}, \cite{Bh1}). Here  we are obtaining  several dilation theory results
     in a much weaker framework with very little structure. We assume that the reader is familiar with basics of dilation theory
 of operators. So we do not dwell on  explaining usefulness of
these ideas nor do we go into nitty gritty of this subject. The basic reference for the subject is the classic book   ``Harmonic analysis of operators on Hilbert space"
\cite{SzF}. A state of the art exposition provided by Orr Shalit \cite{Sha} is also a good place to begin.  We
quote  the results we are mimicking    without giving proofs.

In our setting Hilbert spaces are replaced by sets and bounded operators by arbitrary functions. Injective maps are analogues of isometries and
bijective maps are analogues of unitaries. Direct sums of Hilbert
spaces would be replaced by disjoint unions of sets. In the first section, we begin with an analogue of Wold decomposition of isometries. We study orbits of
injective maps. Like in Operator theory, unilateral shift is the basic model and this is reflected throughout the article. We look at Halmos dilation of contractions and observe that
dilations have three ingredients, namely an embedding, an operator in the bigger space and then a compression. This motivates our constructions. We have simple analogues of
Halmos dilation and Sz. Nagy dilation. We have a notion of defect space. Co-invariant minimal dilations would be parametrized be these defect spaces. Going further, we have analogues
of intertwining lifting theorem and Sarason's lemma.

In the second section, we look at multivariable theory. We have versions of  commuting dilations (such as Ando dilation \cite{Ando}) and non-commuting dilations (such as Bunce\cite{Bunce}, Frazho\cite{Frazho}, Popescu \cite{Popescu} dilations). In Section 3, we have an interesting analogue of Berger-Coburn-Lebow theorem on
commuting isometries.   In the last Section we describe the possibility of extending our results to  semigroups such as ${\mathbb R}^n$ or even to  general monoids. We observe that it is also  possible to have a dilation theory where Hilbert spaces are replaced by vector spaces, bounded operators by linear maps and isometries  by injective linear maps.

We are calling our presentation a caricature as some features of dilation theory are accentuated whereas some other aspects are totally ignored. For instance, we do talk about minimality and so on. Essentially most algebraic structures have been retained whereas analytical concepts such as norm estimates and inequalities have been filtered out.  Naturally, then we do not seem to have analogues of results like von Neumann inequality. We do not think of this as a drawback. It is more interesting than to have all aspects translated in a bijective way.

We denote natural numbers $\{ 1, 2, 3, \ldots \} $ by ${\mathbb N}$ and non-negative integers  $\{0, 1, 2, \ldots \}$ by ${\mathbb Z}_+$.  Similarly ${ \mathbb R} _+= [0, \infty)$. We would use
the notation $A_0\sqcup A_1 $ to indicate disjoint union and for
$A_0\subseteq A$, we write $A_0^c$ to indicate the complement
$A\backslash A_0$, when the larger set $A$ under consideration is
clear.

\section{Single variable dilation theory}

The first result we wish to consider is the familiar Wold Decomposition of isometries. Let $\mathcal
H$ be a Hilbert space and let $V:\mathcal H\to \mathcal H$ be an
isometry. Take ${\mathcal H}_1 = \bigcap _{n=0}^{\infty }
V^n({\mathcal H})$ and ${\mathcal H}_0 = ({\mathcal H }_1)^\perp .$
Then ${\mathcal H}_0, {\mathcal H}_1$ reduce $V$, so that
$V=V_0\oplus V_1$, where  $V_1=V|_{\mathcal H _1}$ is a unitary and
$V_0=V| _{{\mathcal H}_0}$ is a  {\em shift,\/} that is,
$\bigcap_{n=0}^{\infty } V_0^n({\mathcal H}_0 )= \{ 0\}.$ Recall that a subspace $\mathcal W$ is called a
{\em wandering subspace}  for an isometry $V$, if $V^m({\mathcal W})\bigcap V^n({\mathcal W})= \{ 0\}$ for
$m, n\in {\mathbb Z}_+$ with $m\neq n.$

\begin{theorem} (Wold Decomposition) (\cite{SzF} ) Let $\mathcal H$ be a Hilbert space and let $V:\mathcal H\to
\mathcal H$ be an isometry. Then $\mathcal H$ decomposes uniquely as
$\mathcal H= {\mathcal H}_0 \oplus {\mathcal H}_1$ where ${\mathcal
H}_0$ and ${\mathcal H}_1$ reduce $V$, $V|_{{\mathcal H}_1}$ is a
unitary and $\bigcap _{n=0}^{\infty }V^n ({\mathcal H}_0)= \{0\} .$
\end{theorem}

\begin{definition} Let $A$ be a set and let $v:A\to A$ be an injective function.
Then a subset $W$ of $A$ is said to be {\em wandering\/} for $v$ if
$$v^m(W)\bigcap v^n(W)= \emptyset ~~\mbox{for}~~ m\neq n.$$
An injective function $v:A\to A$ is said to be a {\em shift\/} if
$\bigcap _{n=0}^{\infty} v^n(A) = \emptyset .$
\end{definition}

\begin{theorem} Let $A$ be a set and let $v:A\to A$ be an injective
function. Then $A$ decomposes uniquely as $A=A_0\sqcup A_1$ where
$A_0$ and $A_1$ are left invariant by $v$ and $v|_{A_0}$ is a shift
and $v|_{A_1}$ is a bijection.
\end{theorem}

\begin{proof} Take $W= A\backslash v(A)$. If $a\in v^m(W)\bigcap v^{m+k}(W)$
for some $m, k\in {\mathbb Z}_+$, we get $ a=v^m (w_1)=
v^{m+k}(w_2)$ for some $w_1, w_2\in W.$ As $v^m$ is injective, we
get $w_1= v^k(w_2).$ But this is not possible unless $k=0$ as
$w_1\in A\backslash v(A).$ This shows that $W$ is a wandering subset
for $v$. Now take $A_0= \sqcup _{n=0}^{\infty}v^n(W)$ and $A_1=
A_0^c.$ Then $A= A_0\sqcup A_1$. Clearly $v|_{A_0}$ is a shift
and $v|_{A_1}$ is a bijection.

Suppose that $A = A_0^{\prime} \sqcup A_1^{\prime}$ such that $v$ leaves $A_0^{\prime}$ and $A_1^{\prime}$ invariant and $v|_{A_0^{\prime}}$ is a shift and $v|_{A_1^{\prime}}$	is a bijection. Then $A_0^{\prime} = \sqcup_{n = 0}^{\infty}
(v|_{A_0^{\prime}})^n W^{\prime}$ where $W^{\prime} = A_0^{\prime} \setminus v(A_0^{\prime})$. Thus we have that
\begin{align*}
W &= A_0 \setminus v(A_0) = (A_0 \sqcup A_1) \setminus (v(A_0) \sqcup A_1) = A \setminus v(A) \\
&= (A_0^{\prime} \sqcup A_1^{\prime}) \setminus (v(A_0^{\prime}) \sqcup A_1^{\prime}) = A_0^{\prime} \setminus v(A_0^{\prime}) = W^{\prime}.
\end{align*}
Consequently, $A_0 = \sqcup_{n = 0}^{\infty}
v^n (W) = \sqcup_{n = 0}^{\infty}
v^n (W^{\prime}) = A_0^{\prime}$ and hence, $A_1 = A_1^{\prime}$.

\end{proof}

From the Wold decomposition, an isometry on a Hilbert space decomposes as a direct sum of
a pure isometry and a unitary. The pure isometry is just the unilateral shift in $l^2({\mathbb Z}_+)$ with some
multiplicity, whereas the unitary part is understood using the spectral theory. We can carry out a
similar analysis for injective functions. Suppose $v:A\to A$ is an injective function. For
$a,b$ in $A$, write $a\sim b$ if either $a=v^n(b)$ or $v^n(a)=b$ for some $n\in {\mathbb Z}_+.$
Clearly this defines an equivalence relation on $A$. The corresponding equivalence classes are known
as orbits of $v$.

We first classify injective maps which have exactly one orbit. In that direction consider the following examples.

(i)  {\em Cyclic permutations :\/} For $d\in {\mathbb N}$, take ${\mathbb Z}_d =\{0, 1, 2, \ldots , d-1\}$ with addition modulo $d.$ Define $s_d: {\mathbb Z}_d\to {\mathbb Z}_d$ by $s_d(n)=n+1~~\mbox{modulo}~d.$ (Note: ${\mathbb Z}_1=\{0\}$ and $s_0(0)=0.$). Then $s_d$ is bijective and has
single orbit.

(ii) {\em Bilateral translation: \/} Define $s: {\mathbb Z}\to {\mathbb Z}$ by $s(n)=n+1$. Then $s$ is bijective and has single orbit. (We refrain from calling this as a shift as $\bigcap _ns^n({\mathbb Z})\neq \emptyset .$)

(iii) {\em Unilateral translation/shift:\/} Define $s_+: {\mathbb Z}_+\to {\mathbb Z}_+$ by $s_+(n)=n+1.$
Then $s_+$ is a shift with single orbit.

A little bit of thought shows that any injective map with single orbit has to be in bijective correspondence with exactly one of these examples. If $X$ is any set, we define $1_X\times s: X\times {\mathbb Z}\to X\times {\mathbb Z}$ by
$$(1_X\times s)(x,n)= (x, n+1).$$
This will be called as the bilateral translation with multiplicity $X$. In a similar way, we  define $1_X\times s_d$  ($d$-cyclic permutation with
multiplicity $X$) on  $X\times {\mathbb Z}_d$ and $1\times s_+$ ( unilateral shift with multiplicity $X$) on $X\times {\mathbb Z}_+.$

Now suppose $v:A\to A$ is injective. A little bit of thought shows that  the action of $v$ on any of the orbits is in bijective correspondence with one of the examples above. Decomposing $A$ into equivalence classes we see that $(A,v)$ is in bijective correspondence
with $ (B,w)$ where $B= \sqcup _{d\in {\mathbb N}}(X_d\times {\mathbb Z}_d )\sqcup X\times {\mathbb Z}\sqcup X_+\times {\mathbb Z}_+$,
with suitable multiplicity spaces $\{X_d\}_d\in {\mathbb N}, X, X_{{\mathbb Z}_+}$ (some of these terms could be absent)
and $w$ is equal to $1\times s_d$ or $1\times s$ or $1\times s_+$ in respective spaces. Cardinalities of these multiplicity spaces are uniquely determined.

   The most basic dilation theorem for contraction operators on
Hilbert spaces is the following.
\begin{theorem} (Halmos Dilation \cite{Hal},   \cite{Sha}) Let $T$ be a
contraction $(\|T \|\leq 1)$ on a Hilbert space $\mathcal H .$ Then
$U: { \mathcal H }\oplus {\mathcal H} \to {\mathcal H}\oplus
{\mathcal H}$ defined by $$ U= \left[
\begin{array}{cc}
    T&  D_{T^*}\\
  D_T  & -T^*\\
    \end{array}
    \right]$$
 where $D_T= (I-T^*T)^{\frac{1}{2}}$ and $D_{T^*}=(I-TT^*)^{\frac{1}{2}}$, is a unitary.
\end{theorem}

     In other words, every contraction can be enlarged or dilated to a
unitary on a larger space.  If we identify $\mathcal H$
with $\mathcal H \oplus 0$ in $\mathcal H \oplus \mathcal H $, the
action of $U$ on $\mathcal H$ restricted to $\mathcal H$ is same as
the action by $T$, that is,
$$T = P_{\mathcal H}U|_{\mathcal H},$$

    We observe that this dilation has three components to it. First, there
    is the embedding of the original space  in a larger space.
    Then there is a `good' map in the larger space, which is usually called as the dilation. This is followed
    by a compression (or projection) to the range of the original space in the larger space. The last map is usually an idempotent. In probability theory
    contexts this map is a conditional expectation map. This is true of most dilation theory results.

Here is a simple Halmos type dilation for functions. Let $h:A\to A$
be a function. Now we want a bijection (analogue of unitary) on two copies of $A$, that is, on $B =
A\times \{0,1\}$, which could be called as dilation.
Indeed define $i: A\to B$ by $i(a) = (a,0)$ and $u:B\to B$ by $u(a,
m)= (a, 1-m), ~~(a,m)\in B$. Then $i$ is injective and $u$ is
bijective.
 Further define
$p:B\to B$ by
$$p(a,m) = \left\{ \begin{array}{ll}
(a,0) & \mbox {if} ~~m=0\\
(h(a), 0) & \mbox {if}~~m=1.\end{array}\right.
$$
Then $p$ is an idempotent $(p^2=p)$ and the range of $p$ is $i(A).$ The quadruple  $(B,i,u, p)$ has the property $$i(h(a))= p(u(i(a)))~~\forall a\in A.$$

Let us recall the Sz. Nagy dilation (See \cite{SzF}) which is perhaps the most famous dilation of all. Here we
are mentioning the `isometric dilation'. The isometry can be further dilated to a unitary.

\begin{theorem}
Let $T$ be  a contraction on a Hilbert space ${\mathcal H}$. Then there exists a Hilbert space ${\mathcal K}$ containing ${\mathcal H}$ with
an isometry $V: {\mathcal K}\to {\mathcal K}$ such that
$$T^n = P_{\mathcal H}V^n|_{\mathcal H}~~~\forall n\in {\mathbb Z}_+$$
Furthermore, $({\mathcal K}, V)$ can be chosen such that
$${\mathcal K}= ~~\overline {\mbox{span}}\{ V^nh: h\in {\mathcal H}, n\in {\mathbb Z}_+\},$$
and in such a case the pair $({\mathcal K}, V)$ is unique up to unitary equivalence.

\end{theorem}

Now we want to get a Sz.\@ Nagy dilation for functions. Here is the basic definition. Here and elsewhere for any function $h$, $h^0$ is taken as the
identity function.

\begin{definition}
Let $A$ be a non-empty set and let $h:A\to A$ be a function. An
{\em injective power dilation\/}  (or simply a dilation) of $h$, is a
quadruple $(B, i,v, p)$ where $B$ is a set, $i:A\to B$, $v:B\to B$
are injective and $p:B\to B$ is an idempotent with $p(B)=i(A)$.
Further,
$$i(h^n(a))= pv^n(i(a))~~~\forall a\in A, n\in \mathbb{Z}_+.$$
 Any such dilation is said to be {\em minimal\/}
if $B= \bigcup _{n=0}^{\infty } v^n(i(A))$
Two dilations $(B,i,v,p)$ and $(B',i',v',p')$ are said to be bijectively equivalent if there exists a bijection $u:B\to B'$ such that
$i'=ui,$ $v'=uvu^{-1},$ and $p'= upu^{-1}.$

\end{definition}

\begin{theorem} \label{standard} Every function $h:A\to A$ admits a minimal injective
power dilation.
\end{theorem}

\begin{proof} The construction is simple. Take
\begin{eqnarray*}
B&=& A\times {\mathbb Z}_+;\\
i(a)&=& (a,0), ~~a\in A;\\
v(a,m)&=& (a,m+1),~~(a,m)\in B;\\
p(a,m)&=& (h^m(a), 0),~~(a,m)\in B.
\end{eqnarray*}

 It is easily verified that $(B,i,v,p)$ is an injective power dilation of $h$ and it is minimal.
\end{proof}

We will call the dilation constructed in this theorem as the {\em
standard dilation\/} of $h$. Here there is no uniqueness statement.
Although the construction gives technically a minimal dilation, it
may not be optimal, for instance when $h$ is injective there is
actually no need to enlarge the space. To mitigate this to some
extent we make the following definition.

\begin{definition} Let $h:A\to A$ be a function. A subset $D$
of $A$ is said to be a {\em defect\/} space for $h$ if $h|_{D^c}$ is
injective (Here $D^c= A\backslash D$ is the complement of $D$). A defect space $D$
is said to be a minimal defect space if there is no proper subset of
$D$ which is a defect space for $h$.
\end{definition}

The defect space measures as to how far the map is from being injective. In this definition, it is important to note that $h$ need not leave
$D^c$ invariant. We observe that if a family $\{ D_j: j\in J\}$ of  defect spaces of $h$, forms
a chain under inclusion, that is, for any $j,k \in J$ either $D_j\subseteq D_k$ or $D_k\subseteq D_j$, then
$\bigcap _{j\in J}D_j$ is also a defect space.  Now it
follows by Zorn's lemma that every function admits a minimal defect space. A minimal
defect space is empty if and only if the function is injective.

Let $h:A\to A$ be a function and let $D$ be a defect space of $h$.
Define $(B_D, i_D, v_D, p_D)$ by taking
\begin{eqnarray*}
B_D & = &(D^c\times \{0 \})\bigcup (D\times {\mathbb Z}_+);\\
i_D(a)&= &(a,0),~~a\in A;\\
v_D(a, m)&= & \left\{
\begin{array}{ll}
    (h(a), 0),& (a,m) \in D^c\times \{0\} ;\\
  (a, m+1),  & (a,m )\in D\times {\mathbb Z}_+;\\
    \end{array}
    \right. \\
    p_D (a,m)&= &(h^m(a), 0),~~(a,m)\in B_D.
    \end{eqnarray*}
 Then it is  seen
that $(B_D, i_D, v_D, p_D)$ is a minimal injective dilation. This is same as the standard dilation when $D=A.$


The following definition is motivated by the fact that every minimal isometric dilation $V$ of
a contraction $T$ on a Hilbert space ${\mathcal H}$ leaves ${\mathcal H}^{\perp }$ invariant. It
is well-known that such a property is important while considering dilations of tuples of operators and also
in dilation theory of completely positive maps. This property is called  `coinvariance' in \cite{Ar}, `regular' in \cite{Bh2}  and co-increasing in \cite{ShSk}.
We follow Arveson here, as his terminology looks most appropriate in the present situation.

\begin{definition}
An injective dilation $(B,i,v,p)$  of a function $h:A\to A$ is said to be {\em co-invariant}
if $v( i(A)^c)\subseteq i(A)^c.$
\end{definition}
It may be noted that the dilation $(B_D, i_D, v_D, p_D)$  we constructed above  using a defect space $D$ is co-invariant.
Here is an example of a dilation which is not co-invariant.
\begin{example}
Let $A = \{1 \}$ and consider the function $h = Id_A$. Let $B = \{1, 2 \}$ and define $v : B \rightarrow B$ by $v(1) = 2, v(2) = 1$. Let $i: A \rightarrow B$ be the embedding given by $i(1) = 2$. Define $p : B \rightarrow B$ by $p(1) = p(2) = 2$. Then clearly $p$ is an idempotent with $p(B) = \{2 \} = i(A)$. Further note that $p v^n i(1) = 2 = i(h^n(1))$ for all $n \in \mathbb{Z}_{+}$. Thus, $(B, i, v, p)$ is an injective dilation for $h$ which, obviously, is minimal. But this dilation is not co-invariant as $v(i(A)^c) = v(\{1\}) = \{2 \} = i(A)$.

\end{example}

\begin{theorem}
Let  $(B,i,v,p)$ be a co-invariant injective dilation   of a function $h:A\to A$. Then $D=D(B,i,v,p)= \{ a\in A:
v(i(a))\notin i(A) \} $ is a defect space for $h$. Moreover, $(B, i, v, p)$ is in bijective correspondence with
$(B_D, i_D, v_D, p_D).$

\end{theorem}

\begin{proof}
 Let $a, b \in D^c$ be such that $h(a) = h(b)$. Then $v(i(a)), v(i(b)) \in i(A)$ and hence, $v(i(a)) = p v(i(a)) = i(h(a)) =
i(h(b)) = p v(i(b)) = v(i(b))$ and since, $i, v$ are injective, it follows that $a = b$.

Define a map $\psi : B_D \rightarrow B$ by $\psi((a, m)) = v^m(i(a))$ for $(a, m) \in B_D$. We first show that $\psi$ is a bijection. Let $(a, m), (b, n) \in B_D$ be such that $v^m(i(a)) = v^n(i(b))$. Note that if both $a, b \in D^c$, then we must have that $m = n = 0$. If one of $a, b$ belongs to $D$ while the other one in $D^c$, say $a \in D^c$ and $b \in D$, then, of course, $m = 0$ so that $i(A) \ni i(a) = \psi((a, 0)) = \psi((b, n)) = v^n(i(b))$. Now if $n > 0$, the facts that $b \in D$ and the dilation is co-invariant together imply that $v^n(i(b)) \not\in i(A)$ which leads to a contradiction and thus, in this case, we have that $m = n = 0$. Finally, if $a, b$ both belong to $D$ and if we assume that $m \neq n$, say, without loss of generality $m > n$, then $v^m(i(a)) = v^n(i(b))$ would imply that $v^{m-n}(i(a)) = i(b) \in i(A)$. But, as before, it follows from the facts $a \in D$, co-invariance of the dilation and $m-n >0$ that $v^{m-n}(i(a)) \not\in i(A)$ and thus, we arrive at a contradiction. Thus, from $\psi((a, m)) = \psi((b, n))$ we obtain that $m = n$, from which, using injectivity of $v$ and $i$, it follows at once that $a = b$, proving injectivity of the map $\psi$.

 We now show that $\psi$ is surjective. Let $x \in B$. If $x \in i(A)$, say, $x = i(a)$ for some $a \in A$, then clearly $x = \psi((a, 0))$. If $x \notin i(A)$, we assert that there is $(a, m) \in D \times \mathbb{N}$ such that $\psi((a, m)) = x$. Since, by minimality of the dilation, $B = \cup_{n = 0}^{\infty}v^n(i(A))$, $x = v^n(i(a))$ for some $n > 0, a \in A$. If $a \in D$, then $x = \psi((a, n))$ and we are done. If $a \notin D$, we must have that $n > 1$ and let $k$ be the largest integer, $1 \leq k < n$, such that $v^k(i(a)) \in i(A)$ and $v^{k+1}(i(a)) \notin i(A)$. Let $v^k(i(a)) = i(b)$ where $b \in A$. Clearly, $b \in D$ and $\psi((b, n-k)) = v^{n-k}(i(b)) = v^n(i(a)) = x$. Therefore, $\psi$ is surjective and hence, bijective. It is easy to verify that
\begin{align*}
v \psi = \psi v_D, \ p \psi = \psi p_D, \ \mbox{and} \ \psi i_D = i,
\end{align*}
and we leave these verifications to the reader. Thus  $(B, i, v, p)$ is in bijective correspondence with $(B_D, i_D, v_D, p_D)$.

\end{proof}

A contraction $T$ on a Hilbert space $\mathcal H$ is said to be pure
if $(T^*)^n$ converges in strong operator topology to 0. It is well-known that minimal
isometric dilation of a pure contraction is a shift. We wish to have
a similar theorem in our set up. But we do not seem to have an exact
analogue of pureness. An isometry $V$ is pure if and only if it is a
shift that is $\bigcap _{n=0}^{\infty }V^n(\mathcal H )=\{0\}.$  For
contractions $T$ we know that $\bigcap _{n=0}^{\infty }T^n(\mathcal
H )=\{0\}$, then $T$ is pure, but the converse is not true. For instance if $T=tI$ is a scalar with
$0<t<1$, then $T$ is pure but  $\bigcap _{n=0}^{\infty }T^n(\mathcal
H )={\mathcal H}.$

\begin{theorem}
Let $h:A\to A$ be a function and let $(B,i,v,p)$ be a co-invariant, minimal injective dilation
of $h$. Let $D=\{ a\in A: v(i(a))\notin i(A)\}$ be the associated defect space. Then
$v$ is a shift if and only if $\bigcap _{n=0}^{\infty }h^n (D^c)=\emptyset .$
If $h$ is a function such that $\bigcap _{n=0}^{\infty}h^n
(A)=\emptyset $, then for every (not necessarily co-invariant)
minimal injective dilation $(B,i,v,p)$ of $h$, $v$ is a shift.
\end{theorem}

\begin{proof}
 By the previous theorem, $(B,i,v,p)$ is in bijective correspondence with $(B_D, i_D, v_D, p_D).$
Now if $a_0\in $  $ \bigcap _{n=0}^{\infty }h^n (D^c)$, as $v(i(a))= i(h(a)) $ for $a\in D^c,$
$i(a_0)\in \bigcap _{n=0}^{\infty }v^n i(D^c)\subseteq \bigcap _{n=0}^{\infty }v^n (B)$ Therefore $v$ is not a
shift. Conversely suppose $x\in \bigcap _{n=0}^{\infty }v^n (B)$ or equivalently $x\in \bigcap _{n=0}^{\infty }v_D^n (B_D).$
Recall that $B_D= (D^c\times \{0 \})\bigcup (D\times {\mathbb Z}_+).$  If $x=(a,0)$ with $a\in D^c$, we see that $a\in
\bigcap _{n=0}^{\infty }h^n (D^c)$ and we are done. Now if $x=(a, m)$ for some $a\in D$ and $m\in {\mathbb Z}_+$. For $n\geq 1$,
as $x\in v^{n+m}(B)$, from looking at the action of $v_D$, we see that there exist $a_j\in D^c$ for $n\geq j\geq 1$, such that
$$v_D(a_n, 0)= (a_{n-1}, 0), v_D(a_{n-1},0)=(a_{n-2},0),\ldots , v(a_2, 0)= (a_1, 0)$$
$$v_D(a_1,0)= (a,0), v_D(a,0)=(a, 1), v_D(a,1)=(a,2), \ldots , v_D(a, m-1)= (a,m).$$
Further as $h$ is injective on $D^c$, $a_j\in D^c$ for $n\geq j\geq 1$ are uniquely determined and
$h(a_n)=a_{n-1}, h(a_{n-1})=a_{n-2}, \ldots h(a_2)=a_1, h(a_1)=a.$ In particular $a_1= h^{n-1}(a_n)$.
So we get $a_1\in h^{n-1}(D^c).$ As this holds for every $n$, we get $\bigcap _{n=0}^{\infty }h^n (D^c)\neq\emptyset .$
This proves the first part.

Now for the second part, assume  $\bigcap _{n=0}^{\infty}h^n
(A)=\emptyset $.  Let $x\in \bigcap _{n=0}^{\infty }v^n(B).$ We know that $p(x) \in
i(A).$  As $i$ is injective there exists unique $a\in A$ such that
$p(x)= i(a).$  Now for any $n\in {\mathbb Z}_+$, as $x\in v^n(B)$,
there exists some $b_n\in B$ such that $x=v^n(b_n)$. By minimality
of the dilation,  $b_n$ has the form $b_n= v^k(i(a_n))$ for some
$k\in {\mathbb Z}_+$ and $a_n\in A.$ So $x=v^{n+k}(i(a_n))$. Then by
the dilation property, $i(a)= p(x)= i(h^{n+k}(a_n)).$ Therefore for
every $n$ there exists $k\geq 0$ such that $a\in h^{n+k}(A).$ Hence
$a\in \bigcap _{n=0}^{\infty }h^n(A).$ This contradicts purity of
$h$.
\end{proof}

Here is an example of a dilation which is not co-invariant to illustrate the second part of the previous theorem.

\begin{example}
Let $A= {\mathbb Z}_+$ and let $h:A\to A$ be defined by $h(n)=n+1$. Clearly $h$ is pure. Take $B= {\mathbb Z}_+\times \{0,1\} .$
Define $i(a)=(a,0)$, $v(a,0)=(a,1), v(a,1)=(a+2, 0)$ and $p(a,0)=(a,0), p(a,1)=(a+1,0).$ It is clear that $(B, i, v,p)$ is a
minimal injective dilation of $h$ which is not co-invariant.  Nevertheless  $v$ is a shift with multiplicity $2$.

\end{example}

The intertwining lifting theorem of Sz. Nagy and C. Foias is well-known. The commutant lifting theorem is a special case of this. This result is known to have
several applications in interpolation theory and control theory (\cite{FF},\cite{Pa}).

\begin{theorem} (Intertwining lifting theorem \cite{NaFo})  Let $T_1, T_2$ be contractions on Hilbert spaces
${\mathcal H}_1, {\mathcal H}_2$ respectively. Let $V_1, V_2$ acting
on ${\mathcal K}_1, {\mathcal K}_2$ be minimal isometric dilations
of $T_1, T_2$ respectively. Suppose $S: {\mathcal H}_2\to {\mathcal
H}_1$ is a bounded operator such that $T_1S=ST_2$. Then there exists
a bounded operator $R:{\mathcal K}_2\to {\mathcal K}_1$ such that
$V_1R=RV_2$, $P_{{\mathcal H}_1}R|_{{{\mathcal H}_2}^\perp}=0$,
$P_{{\mathcal H}_1}R|_{{\mathcal H}_2}= S$ and $\|R\|=\|S\|.$
Conversely if $R:{\mathcal K}_2\to {\mathcal K}_1$ is a bounded
operator such that $V_1R=RV_2$, and $P_{{\mathcal
H}_1}R|_{{{\mathcal H}_2}^\perp}=0$,  then $S:{\mathcal H}_2\to
{\mathcal H}_1$ defined by $S= P_{{\mathcal H}_1}R|_{{\mathcal
H}_2}$ satisfies $T_1S=ST_2.$

\end{theorem}

Here is an analogue of this theorem.

\begin{theorem}

Let $h_1:A_1\to A_1$ and $h_2:A_2\to A_2$ be two functions. Let
$(B_j, i_j, v_j, p_j)$ be standard dilation of $h_j$ for $j=1,2.$
Suppose $s:A_2\to A_1$ is a function such that $sh_2=h_1s$. Then
there exists a function $r:B_2\to B_1$ such that $rv_2=v_1r,$ $rp_2=p_1r$  and $ri_2=i_1s$.  Conversely if
$r:B_2\to B_1$ is a function satisfying  $rv_2=v_1r$ and $rp_2=p_1r,$ then there exists unique function $s:A_2\to A_1$
satisfying $ri_2=i_1s$ and $sh_2=h_1s.$

\end{theorem}
\begin{proof}
Define $r:B_2\to B_1$ by $r(a,m)= (s(a), m)$ for $(a,m)\in B_2.$ Recall
$v_j(a,m)= (a, m+1)$ for $(a,m)\in B_j$ and $i_j(a)= (a, 0)$, for $j=1,2 $.
 Clearly $rv_1(a,m)= r(a,m+1)= (s(a), m+1)= v_2r(a,m).$ Similarly $rp_2(a,m)=r(h_2^m(a),0)=(sh_2^m(a),0)
=(h_1^ms(a),0)=p_1((s(a), m)=p_1r(a,m)$ and $ri_2(a)=r(a,0)=(s(a),0)=i_1s(a).$

For the converse, consider any $a\in A_2.$ As $rp_2=p_1r$, $rp_2(a,0)\in i_1(A_1).$ As $i_1$ is injective there exists unique $s(a)\in A_1$ such that
$rp_2(a,0)= (s(a), 0)$. This defines $s:A_2\to A_1$  satisfying $ri_2=i_1s.$ Now $rp_2v_2=p_1rv_2=p_1v_1r$ and $rp_2v_2(a,0)=rp_2(a,1)=r(h_2(a),0)=
(sh_2(a),0)$ and $p_1v_1r(a,0)=p_1v_1(s(a),0)=p_1(s(a),1)=(h_1(s(a)),0)$, and hence $sh_2=h_1s.$

\end{proof}

We can generalize this intertwining lifting for dilations with defect spaces as follows.

  \begin{theorem} For each $j = 1, 2$, let $h_j : A_j \rightarrow A_j$ be a function and let $D_j$ be a defect space for $h_j$ and consider the injective minimal dilation $(B_{D_j}, i_{D_j}, v_{D_j}, p_{D_j})$ of $h_j$. Suppose that $s : A_2 \rightarrow A_1$ be a function such that $s(D_2^c) \subseteq D_1^c, s(D_2) \subseteq D_1$, and $s h_2 = h_1 s$. Then there exists a function $r : B_{D_2} \rightarrow B_{D_1}$ such that $r v_{D_2} = v_{D_1} r, \ r p_{D_2} = p_{D_1} r$ and $r i_{D_2} = i_{D_1} s$. Conversely, if $r : B_{D_2} \rightarrow B_{D_1}$ is a function satisfying $r v_{D_2} = v_{D_1} r, \ rp_{D_2} = p_{D_1} r$, then there exists a unique function $s : A_2 \rightarrow A_1$ satisfying $r i_{D_2} = i_{D_1} s$ and $s h_2 = h_1 s$. Moreover, $s(D_2^c) \subseteq D_1^c$.
 	\end{theorem}
 	
 \begin{proof} Similar to the
  proof of previous Theorem.
  \end{proof}

  Here is a special case of the famous Sarason's Lemma (See \cite{Sar}, \cite{Sha}).

\begin{theorem} (Sarason's Lemma \cite{Sar}) Let $\mathcal H$ be a closed subspace of a Hilbert space ${\mathcal K}$. Let
$V:{\mathcal K}\to {\mathcal K}$ be a bounded operator. Define $T:{\mathcal H}\to {\mathcal H}$ by
$$T= P_{\mathcal H}V|_{\mathcal H}.$$
Then $V$ is a power dilation of $T$, that is, $T^n =P_{\mathcal H}V^n|_{\mathcal H}$ for all $n\in {\mathbb Z}_+$ if and only if
there exist two closed  subspaces ${\mathcal M}\subseteq {\mathcal N} \subseteq {\mathcal K}$, invariant under $V$ such that
$${\mathcal H}= {\mathcal N}\ominus {\mathcal M}.$$
\end{theorem}

Now we present a similar result for injective maps.
\begin{theorem}
Let $v:B\to B$ be an injective function and let $A\subset B$. Suppose $h:A\to A$ is a function such that $h(a)=v(a)$ for $a\in A$ with
$v(a)\in A.$ Suppose $A=A_2\backslash A_1$ where $A_1, A_2$ are invariant under $v$. Then there exists $p:B\to B$ such that $p^2=p$, $p(B)=A$ and
\begin{equation}\label{p}
pv^n(a)= h^n(a)~~\forall n\in {\mathbb Z}_+, a\in A.
\end{equation}

\end{theorem}

\begin{proof}
All we need to show is that $p$ defined by (\ref{p}) on $\bigcup _{n\in {\mathbb Z}_+}v^n(A)$  is well-defined. Then it can be arbitrarily extended to some idempotent on $B$.
Suppose $v^{n}(a_1)= v^{n+m}(a_2)$ for some $a_1, a_2$ in $A$ and $n,m \in {\mathbb Z}_+ .$ Then as $v$ is injective, $a_1=v^m(a_2).$ Note that $A\bigcap A_1=\emptyset $.
As $A_1$ is left invariant under $v$, if $v^k(a_2)\in A_1$ for some $k$, $1\leq k\leq m$, then we get a contradiction as $a_1\in A$. Consequently $v^k(a_2)\in A$, and so
$v^k(a_2)= h^k(a_2)$ for $1\leq k\leq m.$ Hence $a_1= h^m(a_2)$ and $h^n(a_1)=h^{n+m}(a_2).$
\end{proof}
The converse of this theorem is false. Here is a simple example to show this.


\begin{example}
Let $B = \mathbb{Z}_+$, the set of non-negative integers, and let $v : B \rightarrow B$ be the injective function defined by $v(0) = 1, v(1) = 0$ and $v(n) = n, n \geq 2$. Let $A = \mathbb{N}$ and let $h : A \rightarrow A$ be the identity function. Clearly, $h(n) = v(n)$ for $n \in A$ with $v(n) \in A$. Further, if we define $p : B \rightarrow B$ by $p(0) = 1$ and $p(n) = n$ for $n \in \mathbb{N}$, then clearly $p$ is an idempotent with $p(B) = A$ and obviously, $pv^m(n) = h^m(n)$ for every $m \geq 0$ and $n \in A$. Suppose $A = A_2 \setminus A_1$ for some subsets $A_1, A_2$ of $B$. Note that there are only two choices, namely, $A_2 = B, A_1 = \{0 \}$ and $A_2 = A, A_1 = \emptyset$. But clearly neither $A$ nor $\{0\}$ is $v$-invariant.

\end{example}

We know that Sz. Nagy dilation allows dilation of contractions to unitaries and not just isometries. In this
context, we have the following result.

\begin{theorem}
Let $h:A\to A$ be a function. Then there exists $(C, i, u, p)$ where $C$ is a set, $i:A\to C$ is injective,
$u:C\to C$ is a bijection, $p:C\to C$ is an idempotent with $p(C)=p(A)$ such that
$$i(h^n(a)) = pu^n(i(a))~~a\in A, n\in {\mathbb Z}_+,$$
and $C= \bigcup_{n\in {\mathbb Z}}u^n((i(A)).$
\end{theorem}
\begin{proof}
Take $C= A\times {\mathbb Z}, i(a)= (a,0), u(a,n)=(a, n+1), p(a,n)= (h^n(a), 0)$ for $n\in {\mathbb Z}_+$
and $p(a,n)= (a,0)$ for $n<0.$

\end{proof}
This result is not quite satisfactory as the action of $p$ on $(a,n)$ with $n<0$ is rather arbitrary.

\section{Multivariable dilation theory}

  One would like to dilate commuting contractions on a Hilbert
space to commuting isometries. In this context the following theorem
has received a lot of attention.

\begin{theorem} (Ando Dilation(\cite{Ando}) Let $T_1, T_2$ be contractions on a
Hilbert space $\mathcal H$ satisfying $T_1T_2=T_2T_1.$ Then there
exists a Hilbert space $\mathcal K$ containing ${\mathcal H}$ as a
subspace with a pair of commuting isometries $V_1, V_2$ on $\mathcal
K$ such that
$$T_1^{n_1}T_2^{n_2}= P_{\mathcal H}V_1^{n_1}V_2^{n_2}|_{\mathcal
H}~~~\forall n_1, n_2\in {\mathbb Z}_+.$$
\end{theorem}

It is well-known that this theorem can not be extended to triples or
to general $d$-tuples with $d\geq 3.$ In contrast, we have the following theorem.

\begin{theorem} Let $J$ be an index set and let $\{ h_j: j\in
J\}$ be a commuting family of functions on a set $A.$ Then there exists a quadruple
$(B, i, (v_j)_{j\in J}, p)$ where $B$ is a set, $i:A\to B$ is an
injective map, $v_j:A\to B$, $j\in J$  is a commuting family of
injective maps, $p:B\to B$ is an idempotent, such that
$$
i(h_{j_1}h_{j_2}\ldots h_{j_k}(a)) = p(v_{j_1}v_{j_2}\ldots
v_{j_k})(i(a)),~~~j_1, j_2, \ldots, j_k\in J, a\in A$$ and $B=
\{v_{j_1}v_{j_2}\ldots v_{j_k}(i(a)): j_1, j_2, \ldots, j_k\in J,
a\in A\}.$
\end{theorem}

\begin{proof}
Take $B= A\times {\mathbb Z}_{+,0}^{J}$, where ${\mathbb
Z}_{+,0}^{J}$ denotes the space of functions from $J$ to ${\mathbb
Z}_+$ taking 0 as their values except for finitely many points. Then
using multi-index notation, for $\alpha \in {\mathbb Z}_{+,0}^J$,
$h^{\alpha }$ is defined as $h^{\alpha } = \Pi _{j\in J} h_j^{\alpha
(j)}.$  Here the product means composition of powers of $h_j$'s and
since they form a commuting family, the order of composition does not
matter. Define $i:A\to B$ by $i(a) = (a, 0)$ (Here $0$ is the zero
function). For $j\in J$, let $\delta _j$ be the function $\delta
_j(k) = \delta _{jk} .$ Define $v_j:B\to B$ by $v(a, \alpha ) = (a,
\alpha +\delta _j)$ and $p(a, \alpha )= (h^\alpha (a) , 0).$

It is not difficult to verify that $(B, i, (v_j)_{j\in J}, p)$ is a
minimal commuting dilation of $\{ h_j: j\in J\}.$

\end{proof}

In the previous theorem also we have not bothered about optimality.

\begin{definition} A commuting injective dilation $(B, i,(v_j)_{j\in J}, p)$
 of a commuting family of functions $\{h_j: j\in J\}$ on some set
$A$ is said to be {\em co-invariant\/} if every $v_j$ leaves $i(A)^c$
invariant.
\end{definition}

To ensure commutativity of the dilation, we are forced to assume in the following theorem that
$D^c$ is left invariant by all $h_j,j \in J$. This was not the case in one variable. Moreover, now
we do not have any uniqueness result.

\begin{theorem}
Let $\{h_j: j\in J\}$ be a commuting family of functions on a set
$A$. Let $D$ be a  defect space for each of  $h_j$, that is, $h_j|_{D^c}$
is injective for every $j$ and suppose $D^c$ is left invariant by every $h_j, j\in J.$ Define $(B_D, i_D, (v_{j,D})_{j\in
J}, p_D)$ by
\begin{eqnarray*} B_D &=& D^c\times \{0\}\bigcup D\times {\mathbb Z}_{+, 0}^J;\\
  i_D(a)&= &(a,0), ~~a\in A;\\
v_{j,D}(a, \alpha ) &= & \left\{ \begin{array}{ll} (h_j(a),
0), &(a,\alpha ) \in D^c\times \{0\};\\
 (a, \alpha +\delta _j), & (a, \alpha ) \in D\times {\mathbb Z}_{+,0}^J;\end{array}\right.\\
 p(a, \alpha )& =& (h^{\alpha }(a), 0).
\end{eqnarray*}
Then $(B_D, i_D, (v_{j,D})_{j\in
J}, p_D)$ is a co-invariant, minimal, commuting, injective dilation.

\end{theorem}

\begin{proof}

Clearly, $(v_{j, D})_{j \in J}$ is a family of injective maps on $B_D$.
Now we show that this is a commuting family. Let $(a, 0) \in D^c \times \{0 \}$. Then, $v_{i, D} v_{j, D} ((a, 0)) = (h_i h_j(a), 0) = (h_j h_i(a), 0) = v_{j, D} v_{i, D}(a, 0)$. Let $(a, \alpha) \in D \times \mathbb{Z}^J_{+, 0}$. Then $v_{i, D} v_{j, D}((a, \alpha)) = (a, \alpha + \delta_j + \delta_i) = v_{j, D} v_{i, D}((a, \alpha))$. Finally, note that if $a \in D^c$, then $p_D v_{j_1, D} v_{j_2, D} \cdots v_{j_n, D}(i_D(a)) = p_D v_{j_1, D} v_{j_2, D} \cdots v_{j_n, D}((a, 0)) = p_D(h_{j_1} h_{j_2} \cdots h_{j_n}(a), 0) = (h_{j_1} h_{j_2} \cdots h_{j_n}(a), 0) = i_D(h_{j_1} h_{j_2} \cdots h_{j_n}(a))$ and if $a \in D$, then $p_D v_{j_1, D} \cdots $ $v_{j_n, D}(i_D(a)) = p_D v_{j_1, D} v_{j_2, D} \cdots v_{j_n, D}((a, 0)) = p_D(a, \delta_{j_1} + \delta_{j_2} + \cdots + \delta_{j_n}) = (h_{j_1} \cdots h_{j_n}(a), 0) = i_D(h_{j_1} h_{j_2} \cdots h_{j_n}(a))$. Thus, $(B_D, i_D, (v_{j, D})_{j \in J}, p_D)$ is a commuting injective dilation of $(h_j)_{j \in J}$. Clearly, this dilation is minimal. Observe that $i(A)^c = D \times \mathbb{Z}^J_{+, 0} \setminus \{0\}$ and given $(a, \alpha) \in i(A)^c$ (so that $a \in D$), $v_{j, D}(a, \alpha) = (a, \alpha + \delta_j) \in i(A)^c$. Thus, $(B_D, i_D, (v_{j, D})_{j \in J}, p_D)$ is a co-invariant, minimal, commuting injective dilation of $(h_j)_{j \in J}$.

\end{proof}

For non-commuting operator tuples forming a row contraction, the
following dilation to isometries with orthogonal ranges is
well-known.

\begin{theorem} (\cite{Bunce}, \cite{Frazho}, \cite{Popescu} and \cite{Sha})
Fix $d\geq 1$ and let $\mathcal H$ be a Hilbert space. Let
$T_j:\mathcal H\to \mathcal H$ be bounded operators satisfying
$$T_1T_1^*+T_2T_2^*+\cdots +T_dT_d^*\leq I.$$
Then there exists a Hilbert space $\mathcal K$ containing $\mathcal
H$ as a subspace, with isometries $V_j:\mathcal K\to \mathcal K$
satisfying $$V_i^*V_j= \delta _{ij},$$ $V_j^*(\mathcal H)\subseteq
\mathcal H$ for all $j$, and
$$ T_{j_1}T_{j_2}\ldots T_{j_k}= P_{\mathcal H}V_{j_1}V_{j_2}\ldots
V_{j_k}|{\mathcal H}$$ for $1\leq j_1, j_2, \ldots , j_k\leq d$ and
$k\geq 1.$

\end{theorem}

Our analogue of this theorem requires the following definition.

\begin{definition}  Let $J$ be an index set and let $\{h_j: j\in J\}$ be a
family of functions on a  set $A.$ A quadruple $(B, i, (v_j)_{j\in
J}, p)$, is an {\em injective non-commutative dilation\/} of
$\{h_j:j\in J\}$ if   $B$ is a set, $i:A\to B$ is an injective
function, for every $j$, $v_j:B\to B$ is injective such that $v_j(B)
\bigcap v_k(B) =\emptyset $ for $j\neq k$, $p:B\to B$ is an
idempotent with $p(B)=i(A)$ and
$$i(h_{j_1}h_{j_2}\ldots h_{j_k}(a))= p(v_{j_1}v_{j_2}\ldots
v_{j_k})(i(a))$$ for $j_1, j_2,\ldots , j_k$ in $J$, $k\in
\mathbb{N}$ and $a\in A.$ Such a dilation is said to be {\em
minimal\/} if
$$B= \{ v_{j_1}v_{j_2}\ldots v_{j_k}(i(a)): j_1, j_2, \ldots , j_k \in J, k\in \mathbb {Z}_+, a\in A\}. $$
\end{definition}

\begin{theorem} Every family of functions $\{h_j: j\in J\}$ on a  set
$A$ admits a minimal injective non-commuting dilation.
\end{theorem}

\begin{proof}
Define $$\Gamma ^J= \{ (j_1, j_2, \ldots , j_k):j_1, j_2, \ldots ,
j_k \in J, k\in \mathbb {N}\} \bigcup \{ \omega \}$$ where $\omega $
is the empty tuple or `vacuum'. Take
\begin{eqnarray*}
B& = & A\times \Gamma ^J;\\
i(a)& = & (a, \omega ), a\in A;\\
v_j(a, \omega )& =&  (a, (j) );\\
 v_j(a,(j_1, j_2, \ldots , j_k))& =& (a, (j, j_1, j_2, \ldots , j_k));\\
 p(a, \omega )& =&  (a, \omega );\\
 p(a, (j_1, j_2, \ldots , j_k))&= & (h_{j_1}h_{j_2}\ldots h_{j_k}(a), \omega).\end{eqnarray*}
 Now the verification of claims
made above is easy.
\end{proof}

\begin{definition} Let $\{h_j: j\in J\}$ be a family of functions on
a set $A$. Then a subset $D$ of $A$ is said to be a joint defect
space for $\{h_j: j\in J\}$ if $h_j|_{D^c}$ are injective and
$$h_j(D^c)\bigcap h_k(D^c)= \emptyset ~~\mbox {for} \ j\neq k.$$
\end{definition}

The condition in the previous definition can also be stated as $ H:
J\times A\to A$ defined by $H(j,a)= h_j(a)$ is injective on $J\times
D^c$.

Now suppose $D$ is a joint defect space for $\{ h_j: j\in J\}.$
Consider $(B_D, i_D, (v_{j,D})_{j\in J}, p)$, where $B_D = D^c\times
\{\omega \}\bigcup D\times \Gamma ^J,$ $i_D(a) = (a, \omega ),$
$v_{j,D}(a, \omega)= (h_j(a), \omega)$ if $a\in D^c$ and $v_{j,D}(a,
\omega ) = (a, (j))$ if $a\in D$ and $v_{j,D}(a, (j_1, j_2, \ldots ,
j_k))= (a, (j, j_1, j_2, \ldots , j_k)).$ Also $p(a, \omega )= (a,
\omega )$ and $p(a, (j_1, j_2, \ldots , j_k))= (h_{j_1}h_{j_2}\ldots
h_{j_k}(a), \omega ).$ Then $(B_D, i_D, (v_{j,D})_{j\in J}, p)$ is a
non-commutative injective dilation of $\{h_j: j\in J\}$. We also observe
that if $v_{j,D}(i_D(a))\notin i_D(A)$ then  $v_{k,D}(i_D(a)) \notin i_D(A)$ for every $k \in J$.
This is a property crucial for the next theorem.


\begin{theorem}
 \emph{Let $\{h_j : j \in J \}$ be a family of functions on a set $A$. Let $(B, i, (v_j)_{j \in J}, p)$ be an injective, non-commutative, minimal, co-invariant dilation of $\{h_j : j \in J \}$ and has the property that if for some $a \in A$ and $j \in J$, $v_j(i(a)) \notin i(A)$, then $v_k(i(a)) \notin i(A)$ for every $k \in J$. Set
\begin{align*}
D = \{a \in A : v_j(i(a)) \notin i(A)  \ \mbox{for some} \ j \in J\}.
\end{align*}
Then $D$ is a joint defect space for $\{h_j : j \in J \}$ and $(B_D, i_D, (v_{j, D})_{j \in J}, p_D)$ is bijectively isomorphic to $(B, i, v, p)$.}
\end{theorem}

\begin{proof} Observe that given any $a \in A$, either $v_j(i(a)) \in i(A)$ for every $j \in J$ or $v_j(i(a)) \notin i(A)$ for all $j \in J$. We first show that $D$ is a joint defect space for $\{h_j : j \in J \}$. Let $j \in J$ and let $a, b \in D^c$ be such that $h_j(a) = h_j(b)$. As $a, b \in D^c$, both $v_j(i(a))$ and $v_j(i(b))$ belong to $i(A)$. Thus, $v_j(i(a)) = p v_j(i(a)) = i(h_j(a)) = i(h_j(b)) = p v_j(i(b)) = v_j(i(b))$ and hence, $a = b$. Also for $j \neq k$, $h_j(D^c) \cap h_k(D^c) = \emptyset$ for if $h_j(D^c) \cap h_k(D^c) \neq \emptyset$, say $x \in h_j(D^c) \cap h_k(D^c)$, then $x = h_j(a) = h_k(b)$ for some $a, b \in D^c$ which would imply that $v_j(i(a)) = v_k(i(b))$ and consequently, $v_j(B) \cap v_k(B) \neq \emptyset$, a contradiction. Hence, $h_j(D^c) \cap h_k(D^c) = \emptyset$ for $j, k \in J$ with $j \neq k$.

 Consider the map $\psi : B_D \rightarrow B$ defined by $\psi(a, \omega) = i(a)$ and $\psi(a, \alpha) = v^{\alpha}(i(a))$ for $(a, \alpha) \in B_D$ with $\alpha \neq \omega$. We show that $\psi$ is injective. Let $(a, \alpha), (b, \beta) \in B_D$ be such that $\psi(a, \alpha) = \psi(b, \beta)$. If one of $a, b$, say $a$, belongs to $D$ and the other one, that is, $b$ belongs to $D^c$, then we must have that $\beta = \omega$ and $\alpha \neq \omega$ and hence, $v^{\alpha}(i(a)) = i(b)$. Let $\alpha = (j_1, j_2, \cdots, j_n)$ where $n \geq 1$. As $a \in D$, $v_j(i(a)) \notin i(A)$ for every $j \in J$ and so, in particular, $v_{j_n}(i(a)) \notin i(A)$. Now the co-invariance of the dilation allows us to conclude that $v^{\alpha}(i(a)) \notin i(A)$, leading to a contradiction. Thus, either $a, b \in D$ or $a, b \in D^c$. If both $a$ and $b$ belong to $D^c$, then clearly, $\alpha = \beta = \omega$ from which it follows that $i(a) = i(b)$ and so, $a = b$. Let $a, b \in D$. If $\alpha = \beta = \omega$, then obviously $a = b$. We assert that it can not happen that one of $\alpha, \beta$ is $\omega$ and the other one is different from $\omega$ for if, say $\beta = \omega$ and $\alpha \neq \omega$, then $v^{\alpha}(i(a)) = i(b)$ which is a contradiction since similar argument as before establishes that $v^{\alpha}(i(a)) \notin i(A)$. Let $\alpha = (j_1, j_2, \cdots, j_n)$ and $\beta = (k_1, k_2, \cdots, k_m)$ where $m, n \geq 1$. Then $v_{j_1} \cdots v_{j_n}(i(a)) = v_{k_1} \cdots v_{k_m}(i(b))$. If possible let $n > m$. It follows from $v_j(B) \cap v_k(B) = \emptyset$ for $j \neq k$ that $j_1 = k_1, \cdots, j_m = k_m$ and thus, $i(b) = v_{j_{m+1}} \cdots v_{j_n}(i(a))$, a contradiction. Thus $m = n$ and once again applying the fact that $v_j(B) \cap v_k(B) = \emptyset$ for $j \neq k$ we obtain that $i(a) = i(b)$ so that $a = b$. Thus $\psi$ is injective.

  Next we show that $\psi$ is surjective. Let $ x \in B$. If $x \in i(A)$, say, $x = i(a)$ for some $a \in A$, then $\psi(a, \omega) = i(a) = x$. Assume that $x \notin i(A)$. It follows from from minimality of the dilation that $x = v_{j_1} v_{j_2} \cdots v_{j_n}(i(a))$ where $n \geq 1, a \in A$ and $j_1, \cdots, j_n \in J$.
 Clearly, if $a \in D$, then $\psi((a, (j_1, j_2, \cdots, j_n))) = x$. If $a \in D^c$, $v_{j_n}(i(a)) \in i(A)$ and let $k$ be the smallest positive integer, $1 < k \leq n$, such that $v_{j_k} v_{j_{k+1}} \cdots v_{j_n}(i(a)) \in i(A)$ and $v_{j_{k-1}} v_{j_k} \cdots v_{j_n}(i(a)) \notin i(A)$. Let $v_{j_k} v_{j_{k+1}} \cdots v_{j_n}(i(a)) = i(b)$ where $b \in A$. As $v_{j_{k-1}}(i(b)) \notin i(A)$, $b \in D$ and $\psi((b, (j_1, \cdots, j_{k-1}))) = v_{j_1} \cdots v_{j_{k-1}}(i(b)) = x$. Therefore, $\psi$ is surjective and hence, bijective.

 One can easily verify that
 \begin{align*}
 \psi v_{j, D} = v_j \psi, \  \psi p_D = p \psi, \ \mbox{and} \ \psi i_D = i, \ \mbox{for every} \ j \in J
 \end{align*}
 and consequently, $(B, i, v, p)$ is bijectively isomorphic to $(B_D, i_D, (v_{j, D})_{j \in J}, p_D)$.

\end{proof}

\section{Berger, Coburn and Lebow Theorem}

        A theorem of Berger, Coburn and Lebow (\cite{BCL}) describes the structure of two commuting isometries.  We follow the
        exposition of Maji, Sarkar and Sankar \cite{MSS}. For a recent non-trivial application of this theorem see Bhattacharyya, Kumar and Sau \cite{TB}. The result is as
follows.

Let $V_1, V_2$ be two commuting isometries on a Hilbert space
${\mathcal H}$ and $V=V_1V_2$. Then by Wold decomposition of $V$,
${\mathcal H}= {\mathcal H}_0\oplus {\mathcal H}_1$ decomposing $V$
as $V=V|_{{\mathcal H}_0}\oplus V|_{{\mathcal H}_1}$, where
$V|_{{\mathcal H}_1}$ is a unitary and $V|_{{\mathcal H}_0}$ is a
shift with some multiplicity. So up to unitary isomorphism
${\mathcal H}_0= {\mathcal H}^2 \otimes {\mathcal M}$, and $V|_{{\mathcal H}_0}=
M_z\otimes I_{\mathcal D}$, where $M_z$ is the standard shift
isometry on the Hardy space and ${\mathcal M}$ is a multiplicity
space.

It is not hard to see that ${\mathcal H}_0$ and ${\mathcal H}_1$
reduce $V_1, V_2$ and so they decompose say as $V_1=V_{10}\oplus
V_{11}$ and $V_2=V_{20}\oplus V_{21}$. Of course, this may not be
Wold decomposition of $V_1, V_2$. However $V_{11}$ and $V_{21}$ are
commuting unitaries. Further, $V_{10}$ and $V_{20}$ are commuting
isometries related by a formula as below.

\begin{theorem} (BCL Theorem \cite{BCL})   Under the set up as above, there exists a
projection $P$ on ${\mathcal D}$ and a unitary $U$ on ${\mathcal D}$
such that $V_{10}= S\otimes U^*P+ I\otimes U^*(I-P)$ and $V_{20}=
I\otimes PU+S\otimes (I-P)U.$ Conversely, any pair of a projection
$P$ and a unitary $U$ on $\mathcal D$ would give a commuting pair of
isometries by this formula.
\end{theorem}

 Let $A$ be a non-empty set and let $v:A\to A$ be an
injective map.  Suppose $v$ factorizes as $v=v_1v_2=v_2v_1$ where $v_i:A\to
A$ are injective for $i=1,2.$ Consider the Wold type decomposition of $v$. So take
$A_1= \bigcap _{n = 0}^{\infty }v^n(A)$ and $A_0=A_1^c.$ Using commutativity of $v_1, v_2$, if $a=v^n(a_1)$,
 then $v_1(a)=v^n(a_2)$ where $a_2=v_1(a_1).$ This shows that $A_1$ is invariant under $v_1$. Now if $v_1(a_1)=v^{n+1}(a_2)$
 then $a_1= v^n(a_2)$ where $a_2=v_2(a_1).$ This shows that if $v_1(a_1)\in A_1$, then $a_1\in A_1$. So $A_0$ is also invariant under $v_1.$
 Similarly $A_0$ and $A_1$ are also invariant under $v_2$. Now $v|_{A_1}$ is a bijection. So $A_1=v(A_1)= v_1v_2(A_1)\subseteq v_1(A_1)\subseteq A_1$. Hence
 $v_1,v_2$ are surjective on $A_1$. Consequently, $v_1, v_2$ on $A_1$ are commuting bijections. So, to understand the structure of $v$ it suffices to consider
 $v|_{A_0}$. Hence, without loss of generality here after we assume that $v$ is a shift.

  Take $W= A\backslash v(A)$. As observed before $W$ is a wandering
subset for $v.$  Similarly take $W_i=A\backslash v_i(A)$ for
$i=1,2.$ Now if
 $a\in W$, then either $a\in W_1$ or $a=v_1(a_1)$ for some $a_1$. But then $a_1$ has to be in $W_2$, otherwise we would have $a\in v_1v_2(A).$
Arguing this way from the commutativity of $v_1, v_2$ we get
$$W= W_1\sqcup v_1(W_2)=v_2(W_1)\sqcup W_2.$$
Define $u: W\to W$, that is, from $W_1\sqcup v_1(W_2)$ to
$v_2(W_1)\sqcup W_2$, by $u(w_1)= v_2(w_1)$ and $u(v_1(w_2))= w_2$
for $w_1\in W_1, w_2\in W_2.$ It is easily seen that $u$ is a
bijection with inverse $u^{-1}$ given by $u^{-1}(v_2(w_1))= w_1$ and
$u^{-1}(w_2)= v_1(w_2).$

 From Wold type decomposition, as $v$ is a shift,  $A= \sqcup
_{n=0}^{\infty} v^n(W)= W\sqcup v(W)\sqcup v^2(W)\sqcup \cdots .$
For $n\geq 1$ we extend the definition of $u$ to $v^n(W)$ by taking
$u(v^n(w))= v^n(u(w))$ for any $w\in W.$ Then $u:A\to A$ is a
bijection commuting with $v$. Take $C_0= \sqcup_{n=0}^{\infty }
v^n(W_2)$ and $C_1= \sqcup _{n=0}^{\infty }v^n(v_2(W_1))$, so that
$A= C_0\sqcup C_1$

Moreover, $v_1:A\to A$ is given by 
$$v_1(x) = \left\{ \begin{array}{cl}
u^{-1}(x) & \mbox {if} ~~x\in C_0\\
u^{-1}(v(x)) & \mbox {if}~~x\in C_1.\end{array}\right.
$$
Similarly, $v_2:A\to A$ is given by
$$v_2(x) = \left\{ \begin{array}{cl}
u(x) & \mbox {if} ~~x\in u^{-1}(C_1)\\
u(v(x)) & \mbox {if}~~x\in u^{-1}(C_0).\end{array}\right.
$$
This can be written using  unilateral shift as follows. Take
$B=W\times {\mathbb Z}_+$. Let  $1\times s_+:B\to B$  be the canonical unilateral shift
with wandering space $W$ (identified
with $W\times {0}$). Let $g:A\to B$ defined by $g(v^n(w)) = (w, n)$ for
$w\in W$ and $n\in {\mathbb Z}_+$ is a bijection such that $v=
g^{-1} (1 \times s_+) g.$ Take $s_i= gv_ig^{-1}$ for $i=1,2$. Then
$$s_1(w,n)= \left\{ \begin{array}{ll}(u^{-1}(w),n) & (w,n)\in W_2\times {\mathbb Z}_+;\\
(u^{-1}(w), n+1) & (w,n)\in v_2(W_1)\times {\mathbb Z}_+;\end{array}\right. $$
and
$$s_2(w,n)=  \left\{\begin{array}{ll} (u(w),n) & (w,n)\in W_1\times {\mathbb Z}_+;\\
(u(w), n+1) & (w,n)\in v_1(W_2)\times {\mathbb Z}_+.\end{array}\right. $$
Recall that $W$ decomposes as $W=W_2\sqcup v_2(W_1)$, and  $u:W\to W$ is a bijection such that $u(W_1)=v_2(W_1)=W_2^c.$ In other words we have the following theorem.

\begin{theorem}\label{bcl}
With notation as above, $(A,v,v_1,v_2)$ is in bijective correspondence with $(W\times {\mathbb Z}_+, 1_W\times s_+, s_1, s_2)$ where,
$$s_1 = (u^{-1}\times ~\mbox{id.}~)|_{W_2\times {\mathbb Z}_+}+(u^{-1}\times s_+)|_{W_2^c\times {\mathbb Z}_+};$$
$$s_2= (u\times ~\mbox{id.}~)|_{u^{-1}(W_2^c)\times {\mathbb Z}_+}+(u\times s)|_{u^{-1}(W_2)\times {\mathbb Z}_+}.$$
\end{theorem}
Now we extend our BCL type theorem to  families of maps.
Let $\{ v_j:j\in \{1, 2, \ldots n\} \}$ be a commuting family of
injective maps on a set $A$. Assume that $v=v_1v_2\ldots v_n$ is a
shift. Take $W=A \setminus v(A)$ and $W_j= A \setminus v_j(A)$ for
$j=1,2, \ldots , n.$ Let $S_n$ denote the group of  permutations of $\{ 1,2,
\ldots ,n\}.$ Now for any  $\sigma \in S_n$, $W$ has a decomposition given by
$$W=W_{\sigma (1)}\sqcup v_{\sigma(1)}(W_{\sigma(2)})\sqcup v_{\sigma(1)}v_{\sigma (2)}(W_{\sigma (3)}) \sqcup \cdots
\sqcup v_{\sigma (1)}v_{\sigma (2)}\cdots v_{\sigma (n-1)}(W_{\sigma
	(n)}),$$ and we use the notation $W(\sigma)$ to denote this decomposition of $W$. For $\tau, \sigma \in S_n$, let $u_{\tau}^{\sigma}$ denote the bijection from $W(\sigma)$ to  $W(\tau)$ defined as follows:
\begin{align*}
u_{\tau}^{\sigma}(v_{\sigma (1)}v_{\sigma (2)}\cdots v_{\sigma
	(k-1)}(w))= v_{\tau(1)}v_{\tau(2)}\ldots v_{\tau (r-1)} (w)
\end{align*} for
$w\in W_{\sigma (k)}$, with $r$ being chosen such that $\tau
(r)=\sigma (k)$.
For $k = 0, 1, \cdots, n-1$, let $\sigma_k \in S_n$ be given by
\[\sigma_k(j) = \begin{cases}
k+j, \ \mbox{for} \ 1 \leq j \leq n-k,\\
k-n+j, \ \mbox{for} \ n-k+1 \leq j \leq n.
\end{cases}\] and we consider the family $\{u^{\sigma_k}_{\sigma_{k-1}} : 1 \leq k \leq n \}$ of bijections of $W$ where the notation $\sigma_n$ stands for $\sigma_0$. Using these bijections, we can describe the maps $v_k$ ($1 \leq k \leq n$) on $W$
as follows:
\[v_k = \begin{cases}
u^{\sigma_k}_{\sigma_{k-1}}, \ \mbox{on} \ W \setminus \prod_{i=1, i \neq k}^{n}v_i(W_k),\\
v u^{\sigma_k}_{\sigma_{k-1}}, \ \mbox{on} \ \prod_{i=1, i \neq k}^{n}v_i(W_k).  	
\end{cases}\]
Now, as before, using the fact that $A = \sqcup_{m = 0}^{\infty} v^m(W)$, we may extend the bijections $u^{\sigma_k}_{\sigma_{k-1}}$ of $W$ to bijections of $A$ by setting $u^{\sigma_k}_{\sigma_{k-1}}(v^m w) =  v^m(u^{\sigma_k}_{\sigma_{k-1}} w)$ for $w \in W$ and $m \geq 0$. Consequently,
\[v_k = \begin{cases}
u^{\sigma_k}_{\sigma_{k-1}}, \ \mbox{on} \ \sqcup_{m=0}^{\infty}v^m(W \setminus \prod_{i=1, i \neq j}^{n}v_i(W_k)),\\
v u^{\sigma_k}_{\sigma_{k-1}}, \ \mbox{on} \ \sqcup_{m=0}^{\infty}v^m(\prod_{i=1, i \neq j}^{n}v_i(W_k)).  	
\end{cases}\]
Note that as before, $g^{-1} (1 \times s_+) g = v$. Set $s_k = g v_k g^{-1}$ for $k = 1, 2, \cdots, n$.
We then have the following result.
\begin{theorem}
	With the same notations as in Theorem \ref{bcl}, $(A, v, v_1, v_2, \cdots, v_n)$ is in bijective correspondence with $(W \times \mathbb{Z}_+, 1_W \times s_+, s_1, s_2, \cdots, s_n)$ where
	\begin{align*}
	s_k = (u^{\sigma_k}_{\sigma_{k-1}} \times id.)|_{W_k^{\prime \prime} \times \mathbb{Z}_+} + (u^{\sigma_k}_{\sigma_{k-1}} \times s_+)|_{W_k^{\prime} \times \mathbb{Z}_+}, \ 1 \leq k \leq n;
	\end{align*}
	and $W_k^{\prime} = \prod_{i=1, i \neq k}^n v_i(W_k) = W \setminus (u^{\sigma_{k-1}}_{\sigma_k}(W(\sigma_{k-1})\setminus W_k)), W_k^{\prime \prime} = W \setminus W_k^{\prime}$.	
\end{theorem}

\section{Generalizatons and variations}

We considered dilations of $\{h^n: n\in {\mathbb Z}_+\} $ for a map $h:A\to A.$ Now instead of ${\mathbb Z}_+$ we can
consider ${\mathbb R}_+$ or more general monoids. Dilation theory on general monoids have been considered by many authors, see for instance \cite{ShSk}. Recall that a monoid is a set $S$ with an associative binary operation (say `$.$')
and an identity element (say $1$). The associative operation need not be abelian.  It seems eminently feasible to extend most of what we did in previous sections to this setting.
Here is the standard dilation in this setting.

\begin{theorem}
Let $S$ be a left cancellative monoid. Suppose $A$ is a set and $\{ h_s  : s\in S\} $  is  a family of functions such that
$h_{s.t}=h_s\circ h_t$ for all $s,t\in S$ and $h_1=~\mbox{id.}.$  Then there exists a quadruple $(B,i, \{v_s:s\in S\}, p)$
where $B$ is a set, $i:A\to B$, $v_s:B\to B$ are injective functions, $v_{s.t}=v_s\circ v_t$ for all $s,t $ in $S$, $v_1=~\mbox{id.}$,
$p:B\to B$ is idempotent with $p(B)=i(A)$ such that $$p(v_s(i(a)))=i(h_s(a)) ~~\forall s\in S, a\in A$$
and $B=\bigcup _{s\in S}\{ v_s(i(a)): s\in S, a\in A\}.$
\end{theorem}

\begin{proof}
Take $B=A\times S$. Define $i:A\to B$ by $i(a)= (a,1)$, $v_s(a,t)= (a, s.t)$ and $p(a,s)= (h_s(a), 1).$ The left cancellative property of the monoid $S$
ensures that $v_s$ are isometries.
\end{proof}

 Instead of working with sets and functions, we can try to develop
the dilation theory working with vector spaces and  linear maps. Now bounded operators on Hilbert spaces gets replaced by arbitrary linear maps, isometries by injective linear maps and unitaries by bijective linear maps. Direct sums of Hilbert spaces gets replaced by direct sums of vector spaces.
Here is a formal definition and a sample result.

\begin{definition} Let $A$ be a vector space and let $h:A\to A$ be a
linear map. A quadruple $(B, i, v, p)$ is said to be a minimal
injective linear dilation of $h$ if $B$ is a vector space, $i:A\to B$ is an
injective linear map, $v:B\to B$ is an injective linear map, $p:B\to
B$ is an idempotent linear map with $p(B)=i(A)$, satisfying
$$i(h^n(a)) = p(v^n(i(a))~~~\forall n\in {\mathbb Z }_+, a\in A.$$
Such a dilation is said to be minimal if $B=~
\mbox{span}~\{v^n(i(a)): n\in {\mathbb Z}_+, a\in A\}.$
\end{definition}

\begin{theorem}
Every linear map $h:A\to A$ on a vector space admits a minimal
injective linear dilation.
\end{theorem}

\begin{proof}

Take $B= A^{{\mathbb Z}_+}_0$, the space of functions from ${\mathbb
Z} _+$ to $A$ which take value $0$ at all but finitely many points. It is a
vector space under natural linear operations. Define $i:A\to B$ by
$i(a)(0) = a$ and $i(a)(n)=0$ for $n\neq 0.$ Define $v:B\to B$ by
$v(b)(0)=0$ and $v(b)(n)=b(n-1)$ for $n\geq 1.$ Finally define
$p:B\to B$ by $p(b)= \sum _{n=0}^{\infty}ih^n(b(n)).$ This map $p$ is
well defined as $b(n)=0$ for all but finitely many $n$. It is easy
to see that $(B,i,v,p)$ is a minimal injective linear dilation of $h.$

\end{proof}

We observe that the construction of dilation here is similar to the standard dilation of functions (See Theorem \ref{standard}), however there are certain differences.
Now the addition operation of vector spaces plays a non-trivial role.

 \noindent\textbf{Acknowledgements.} The first author thanks J C
    Bose Fellowship of SERB (India) for financial support. The other two authors are supported by the  NBHM (India) post-doctoral fellowship and the Indian Statistical Institute. We thank T. Bhattacharyya for some useful discussions on the topic.

\end{document}